\documentclass[dvips,sn-mathphys-num]{sn-ppn}

\usepackage{anyfontsize}
\usepackage{graphicx}
\usepackage{multirow}
\usepackage{amsmath,amssymb,amsfonts}
\usepackage{amsthm}
\usepackage{mathrsfs}
\usepackage[title]{appendix}
\usepackage{xcolor}
\usepackage{textcomp}
\usepackage{manyfoot}
\usepackage{booktabs}
\usepackage{algorithm}
\usepackage{algorithmicx}
\usepackage{algpseudocode}
\usepackage{listings}

\newtheorem{theorem}{Theorem}
\newtheorem{proposition}[theorem]{Proposition}
\newtheorem{lemma}[theorem]{Lemma}
\newtheorem{corollary}[theorem]{Corollary}

\newcommand{\R}{\mathbb R}
\newcommand{\N}{\mathbb N}

\newcommand{\be}[1]{\begin{equation}\label{#1}}
\newcommand{\ee}{\end{equation}}
\renewcommand{\(}{\left(}
\renewcommand{\)}{\right)}
\newcommand{\ird}[1]{\int_{\R^d}{#1}\,dx}
\newcommand{\nrm}[2]{\left\|#1\right\|_{\mathrm L^{#2}(\R^d)}}
\newcommand{\nrmone}[2]{\left\|#1\right\|_{\mathrm L^{#2}(\R)}}
\newcommand{\irdg}[1]{\int_{\R^d}{#1}\,d\gamma}
\newcommand{\nrmg}[2]{\left\|#1\right\|_{\mathrm L^{#2}(\R^d,d\gamma)}}
\newcommand{\PP}{\mathsf P}
\newcommand{\HH}{\mathrm{Hess}\,}

\makeatletter
\@namedef{subjclassname@2020}{\textup{2020} Mathematics Subject Classification}
\makeatother
\newcommand{\msc}[1]{\burlalt{https://zbmath.org/classification/?q=cc:#1}{#1}}

\begin{document}
\thispagestyle{empty}
\title[Stability and instability of the log-Sobolev inequality]{\centering{Logarithmic Sobolev inequalities:\\ a review on stability and instability results}}
\author*[1]{\fnm{Giovanni} \sur{Brigati}}\email{giovanni.brigati@ist.ac.at}

\author[2]{\fnm{Jean} \sur{Dolbeault}}\email{dolbeaul@ceremade.dauphine.fr}

\author[3]{\fnm{Nikita} \sur{Simonov}}\email{nikita.simonov@sorbonne-universite.fr}

\affil*[1]{\orgdiv{ISTA}, \orgname{Institute of Science and Technology Austria}, \orgaddress{\street{Am Campus 1}, \city{Klosterneuburg}, \postcode{3400}, \country{Austria}}}

\affil[2]{\orgdiv{CEREMADE}, \orgname{Centre de Recherche en Math\'ematiques de la D\'ecision (CNRS UMR n$^\circ$~7534), PSL University, Universit\'e Paris-Dauphine}, \orgaddress{\street{Place de Lattre de Tassigny}, \postcode{75775} \city{Paris 16}, \country{France}}}

\affil[3]{\orgname{Sorbonne Universit\'e, Universit\'e Paris Cit\'e,} \orgdiv{LJLL}, Laboratoire Jacques-Louis Lions (CNRS UMR n$^\circ$~7598), \postcode{75005} \city{Paris}, \country{France}}

\abstract{In this paper, we review recent results on stability and instability in logarithmic Sobolev inequalities, with a particular emphasis on strong norms. We consider several versions of these inequalities on the Euclidean space, for the Lebesgue and the Gaussian measures, and discuss their differences in terms of moments and stability. We give new and direct proofs, as well as examples and discuss the stability of a logarithmic uncertainty principle. Although we do not cover all aspects of the topic, we hope to contribute to establishing the state of the art.}



\maketitle

\date{\today}

\section{Introduction}\label{Sec:Introduction}

Let $d\geq 1$, and let $d\gamma=\gamma(x)\,dx$, with $\gamma(x)=(2\,\pi)^{-d/2}\,\mathrm e^{-|x|^2/2}$, be the normalized Gaussian proba\-bility measure. The \emph{Gaussian logarithmic Sobolev inequality} on $\R^d$ reads as
\be{LSI-G}
\irdg{|\nabla v|^2}\ge\frac12\irdg{|v|^2\,\log|v|^2}
\ee
for any function $v\in\mathrm H^1(\R^d,d\gamma)$ such that $\nrmg v2=1$. Moreover, by Jensen's inequality, we know that the right-hand side of~\eqref{LSI-G} is non-negative.
Throughout the paper, we will consider only real-valued functions.

If $v$ is a smooth and compactly supported function such that $\nrmg v2=1$, then an elementary computation shows that~\eqref{LSI-G} written for $v$ is equivalent for $u=v\,\sqrt\gamma$ to the \emph{Euclidean logarithmic Sobolev inequality} on $\R^d$,
\be{LSI-E}
\ird{|\nabla u|^2}\ge\frac12\ird{|u|^2\,\log|u|^2}+\frac d4\,\log\big(2\,\pi\,\mathrm e^2\big)
\ee
which, by density, holds for any function $u\in\mathrm H^1(\R^d,dx)$ such that $\nrm u2=1$. However, even if $u$ is smooth and compactly supported, it does not mean that $\ird{|u|^2\,\log|u|^2}$ is uniformly bounded from below, whatever $\|u\|_{\mathrm H^1(\R^d,dx)}$ is.

On $(\R^d,dx)$, one can take advantage of scalings. For any $\lambda>0$, let us consider
\[
u_\lambda(x):=\lambda^{d/4}\,u\big(\sqrt\lambda\,x\big)\quad\forall\,x\in\R^d\,.
\]
Inequality~\eqref{LSI-E} applied to $u_\lambda$ becomes
\be{LSI-E-lambda}
\ird{|\nabla u|^2}\ge\frac1{2\,\lambda}\ird{|u|^2\,\log|u|^2}+\frac d{4\,\lambda}\,\log\big(2\,\pi\,\mathrm e^2\,\lambda\big)
\ee
for any function $u\in\mathrm H^1(\R^d,dx)$ such that $\nrm u2=1$. After optimizing on $\lambda>0$, we obtain the \emph{Euclidean logarithmic Sobolev inequality in scale invariant form};
\be{LSI-S}
\ird{|\nabla u|^2}\ge\frac{\pi\,d\,\mathsf e}2\,\exp\(\frac2d\ird{|u|^2\,\log|u|^2}\)
\ee
for any function $u\in\mathrm H^1(\R^d,dx)$ such that $\nrm u2=1$.

\medskip Logarithmic Sobolev inequalities have a long history. The \emph{Gaussian logarithmic Sobolev inequality}~\eqref{LSI-G} is due to L.~Gross in~\cite{MR420249} and its equivalence with~\eqref{LSI-E} is well-known (see for instance~\cite[Identity~(3)]{MR1132315}), while its scale invariant form~\eqref{LSI-S} appeared in~\cite[Inequality~(2.3)]{MR0109101} in dimension $d=1$ and in~\cite[Theorem~2]{MR479373} for any $d\ge1$. Among earlier related results, one has to quote~\cite{Federbush}. We refer to~\cite[Section~1.3.2]{Villani2008} and also to~\cite{zbMATH01503413,MR3034582,MR3339594} for further background references in information theory and to~\cite{MR1132315} for the equality case, as well as an early stability result. See~\cite{MR3269872,zbMATH06661531,zbMATH06709767,MR3493423,Indrei_2023} and references therein for more recent results and~\cite{MR1845806,Guionnet2002,MR2352327,MR3155209} for related books.
\medskip

A function achieving equality in a given functional inequality, \emph{e.g.}, \eqref{LSI-G}, is called an \emph{optimal function} or an \emph{optimiser}. We indicate with $\mathcal{M}$ the manifold of all optimisers. Optimality in \eqref{LSI-G}, which can also be deduced from~\cite{MR889476}, is characterised as follows in~\cite{MR1132315}: $v$ is optimal if and only if
\[
v \in \mathcal{M} := \big\{ v_{a,b} := a\,\mathrm e^{b\cdot x}\, :\, a \in \mathbb R\,, \, b \in \mathbb R^d \big\}\,.
\]

\medskip The goal of this paper is to review some \emph{stability properties} of Inequalities~\eqref{LSI-G},~\eqref{LSI-E}, \eqref{LSI-E-lambda} and~\eqref{LSI-S}, mostly in strong norms. In the case of~\eqref{LSI-G}, the Gaussian \emph{deficit} is defined~by
\be{delta}
\delta[v]:=\irdg{|\nabla v|^2}-\frac12\irdg{|v|^2\,\log\(\frac{|v|^2}{\nrmg v2^2}\)}
\ee
and we aim either at an \emph{improved inequality} showing that $\delta[v]$ is bounded from below by a functional evaluated on $v$ under a constraint (otherwise~\eqref{LSI-G} would not be optimal), or by a distance to the manifold of optimal functions (see~Theorem~\ref{thm:opti}).

For Sobolev's inequality, the issue was raised by H.~Brezis and E.~Lieb in~\cite{zbMATH03923428}. Slightly earlier, in~\cite{MR778970}, P.-L.~Lions proved a \emph{sequential stability} property: a normalized sequence of optimizing functions $(u_n)_{n\in\N}$ converges in $\dot{\mathrm H}^1(\R^d,dx)$ to an Aubin-Talenti function via the concentration-compactness method, up to the extraction of a subsequence (relative compactness) and after taking advantage of the invariances. Soon after~\cite{zbMATH03923428}, G.~Bianchi and H.~Egnell proved in~\cite{MR1124290} that for some constant $\kappa_d>0$, the deficit associated with Sobolev's inequality is bounded from below by $\kappa_d\,\mathsf d(v,\mathcal M)^2$ where~$\mathsf d$ is the distance induced by $\dot{\mathrm H}^1(\R^d,dx)$ and $\mathcal M$ is the manifold of the \emph{Aubin-Talenti functions}. A lower bound on $\kappa_d$ is known from the recent work~\cite{DEFFL2025}.

For the logarithmic Sobolev inequality, it is thus natural to ask whether there is a \emph{quantitative stability property} for~\eqref{LSI-G}, that is, whether there is some $\kappa>0$ such that
\[
\delta[v]\ge\kappa\,\mathsf d(v,\mathcal M)^2\quad\forall\,v\in\mathrm H^1(\R^d,d\gamma)\,,
\]
where $\mathcal M$ is now the manifold of optimal functions for~\eqref{LSI-G} and investigate for which distance~$\mathsf d$ this stability inequality holds true. Going back to~\cite{MR3269872,MR3271181}, results are know when $\mathsf d$ is a Wasserstein distance. A conditional stability result in $\mathrm L^2$ was obtained in~\cite[Proposition~4.7]{Feo_2016}. The stability inequality is true if $\mathsf d$ is induced by $\mathrm L^2(\R^d,d\gamma)$ according to~\cite{DEFFL2025} but it is not true without additional assumptions if $\mathsf d$ is based on $\mathrm H^1(\R^d,d\gamma)$ as shown for instance in~\cite{kim2018,Indrei_2023,MR4305006}. We shall give details on known stability results in Section~\ref{Sec:Stability} and elaborate on examples of instabilities based on~\cite{MR4305006,kim2018} in Section~\ref{Sec:Instability}, while in Section~\ref{Sec:Examples} we collect relevant facts on $\mathrm{H}^1$ spaces, second-order moments, and entropy functionals. We also emphasize a few differences between~\eqref{LSI-G},~\eqref{LSI-E},~\eqref{LSI-E-lambda}, and~\eqref{LSI-S}. Second moment estimates play a crucial role in many partial results. We would like to draw the attention of the reader to a \emph{logarithmic uncertainty principle} and its stability (see Section~\ref{sss:iiusm}) which seems remarkable.

\section{\texorpdfstring{$\mathrm H^1$}{H1} spaces and logarithmic Sobolev inequalities}\label{Sec:Examples}

Let us start by collecting some observations on the differences between the $\mathrm H^1$ spaces with respect to Lebesgue and Gaussian measures and the consequences for the corresponding forms of the logarithmic Sobolev inequalities on $\R^d$. The space $\mathrm H^1(\R^d,d\gamma)$ is obtained by the completion of smooth and compactly supported functions with respect to the norm $v\mapsto\|v\|_{\mathrm H^1(\R^d,d\gamma)}$ with $\|v\|_{\mathrm H^1(\R^d,d\gamma)}^2=\nrmg{\nabla v}2^2+\nrmg v2^2$, and it is classical in the study of the logarithmic Sobolev inequality, see, \emph{e.g.},~\cite[Def.~3.1.11]{MR2352327}.

\subsection{Integrability and averages in the Euclidean case}

The \emph{Euclidean logarithmic Sobolev inequality}~\eqref{LSI-E} on $\R^d$ can be written for any function $u\in\mathrm H^1(\R^d,dx)$ such that $\nrm u2=1$. This is not enough to prove that $\ird{|u|^2\,\log|u|^2}$ is uniformly bounded from below as shown by the following examples.

\medskip\noindent$\bullet$ \emph{Example 1.} Assume that $d=1$ and let $u$ be a smooth function on $\R$ with compact support in $(0,1)$. Let
\be{Example}
u_n:=\frac1{\sqrt n}\,\sum_{k=0}^{n-1}u(x+k)
\ee
so that $\nrmone{u_n}2=\nrmone u2$ and $\nrmone{\nabla u_n}2=\nrmone{\nabla u}2$ for any $n\ge1$, while
\[
\int_{\R}|u_n|^2\,\log|u_n|^2\,dx=\int_{\R}|u|^2\,\log|u|^2\,dx-\nrmone u2^2\,\log n\to-\infty\quad\mbox{as}\quad n\to+\infty\,.
\]
\noindent$\bullet$ \emph{Example 2.} On $\R^d$, let us consider the function
\[
u(x)=\(1+|x|^2\)^{-\frac d4}\(\log\(2+|x|^2\)\)^{-\frac a2}\quad\forall\,x\in\R^d
\]
for some $a\in(1,2)$. This function is smooth and such that as $|x|\to+\infty$
\begin{align*}
&|x|^2\,|\nabla u(x)|^2\sim\frac{d^2}4\,|u(x)|^2=O\(|x|^{-d}\(\log|x|\)^{-a}\)\,,\\
&|u(x)|^2\,\log|u(x)|^2=O\(|x|^{-d}\(\log|x|\)^{1-a}\)\,.
\end{align*}
One can check that $u\in\mathrm H^1(\R^d,dx)$ is such that
\[
\lim_{R\to+\infty}\int_{|x|<R}|u|^2\,\log|u|^2\,dx=-\,\infty\,.
\]

\medskip It is a natural question to ask under which additional condition on $u\in\mathrm H^1(\R^d,dx)$ one can guarantee that $|u|^2\,\log|u|^2\in\mathrm L^1(\R^d)$. If this is the case, let us observe that we can choose $\lambda>0$ such that $\ird{|u_\lambda|^2\,\log|u_\lambda|^2}=0$ where $u_\lambda:=\lambda^{d/2}\,u(\lambda\cdot)$ as
\[
\ird{|u_\lambda|^2\,\log|u_\lambda|^2}=\ird{|u|^2\,\log|u|^2}+d\,\log\lambda\,\nrm u2^2
\]
uniquely determines $\lambda$. Interestingly, we have a reciprocal result that goes as follows. Let us consider the Gagliardo-Nirenberg-Sobolev inequality
\be{GNS}
\nrm{\nabla u}2^\theta\,\nrm u2^{1-\theta}\ge\mathcal C_{\mathrm{GNS}}(d,p)\,\nrm up\quad\forall\,u\in\mathrm H^1(\R^d,dx)
\ee
where $\theta=d\,(p-2)/(2\,p)$ and $\mathcal C_{\mathrm{GNS}}(d,p)>0$ is the optimal constant. The exponent $p$ is larger than $2$, with the additional restriction that $p\le2\,d/(d-2)$ if $d\ge3$. If $d\ge3$ and $p=2\,d/(d-2)$, then $\theta=1$ and~\eqref{GNS} is the classical Sobolev inequality.
\smallskip\begin{proposition}\label{Prop:1} With this notation and $p$ as above, if $u$ is a smooth and compactly supported function such that $\ird{|u|^2\,\log|u|^2}=0$, then
\[
\ird{\big|\,|u|^2\,\log|u|^2\big|}\le4\,\frac{\(\nrm{\nabla u}2^\theta\,\nrm u2^{1-\theta}\)^p}{(p-2)\,\mathrm e\,\mathcal C_{\mathrm{GNS}}(d,p)^p}\,.
\]
\end{proposition}
\noindent By density, the result of Proposition~\ref{Prop:1} also holds in $\mathrm H^1(\R^d,dx)$.
\begin{proof} A simple optimization shows that
\[
\sup_{t>1}\frac{t^2\,\log t^2}{t^p}=\frac2{(p-2)\,\mathrm e}
\]
for any $p>2$. As a consequence with $t=|u|$, we have
\[
-\int_{|u|\le1}|u|^2\,\log|u|^2\,dx=\int_{|u|\ge1}|u|^2\,\log|u|^2\,dx\le2\frac{\nrm up^p}{(p-2)\,\mathrm e}\,,
\]
which completes the proof by using $\ird{\big|\,|u|^2\,\log|u|^2\big|}=2\int_{|u|\ge1}|u|^2\,\log|u|^2\,dx$ and~\eqref{GNS}.
~\end{proof}
We can deduce a criterion of integrability from~Proposition~\ref{Prop:1}.
\smallskip\begin{corollary}\label{Cor:L1criterion} If $u\in\mathrm H^1(\R^d,dx)\setminus\{0\}$, then
\begin{itemize}
\item[\rm (i)] either for any sequence $(u_n)_{n\in\N}$ of smooth and compactly supported functions on~$\R^d$ such that $\lim_{n\to+\infty}\big(\nrm{\nabla u-\nabla u_n}2^2+\nrm{u-u_n}2^2\big)=0$, we have
\[
\lim_{n\to+\infty}\int_{\R}|u_n|^2\,\log|u_n|^2\,dx=-\,\infty\,,
\]
\item[\rm (ii)] or the function $u$ is such that $|u|^2\,\log|u|^2\in\mathrm L^1(\R^d)$.
\end{itemize}
\end{corollary}
\begin{proof} If (i) does not hold, then one can find a sequence $(u_n)_{n\in\N}$ such that
\[
\lambda_n=\exp\(-\frac1d\,\frac{\ird{|u_n|^2\,\log|u_n|^2}}{\nrm{u_n}2^2}\)
\]
converges to some $\lambda\ge0$. It is then clear that $\widetilde u_n=\lambda_n^{d/2}\,u_n\big(\sqrt{\lambda_n}\cdot\big)$ satisfies the conditions of Proposition~\ref{Prop:1}: $\ird{|\widetilde u_n|^2\,\log|\widetilde u_n|^2}=0$, while we notice that \hbox{$\ird{|\nabla\widetilde u_n|^2}\sim\lambda_n\ird{|\nabla u_n|^2}\to0$} as $n\to+\infty$ if $\lambda=0$. This contradicts~\eqref{LSI-E} applied to $\widetilde u_n$. As a consequence, we have that $\lambda$ is a positive real number such that $(\widetilde u_n)_{n\in\N}$ converges to $\widetilde u=\lambda^{d/2}\,u\big(\sqrt\lambda\cdot\big)$ in $\mathrm H^1(\R^d,dx)$. By Proposition~\ref{Prop:1} and Fatou's lemma, $|\widetilde u|^2\,\log|\widetilde u|^2\in\mathrm L^1(\R^d)$ and, as a consequence, $|u|^2\,\log|u|^2\in\mathrm L^1(\R^d)$.\end{proof}

\subsection{Integrability and moments. Gaussian and Euclidean cases}

It turns out that a second moment condition is a sufficient condition to guarantee that $|u|^2\,\log|u|^2\in\mathrm L^1(\R^d)$. Here is a statement and a proof of this classical result that apparently goes back to~\cite[Sec.~7]{MR1555365} and~\cite{MR26286,MR26286b}. Similar estimates can be found, \emph{e.g.}, in~\cite{Kubo2019} or~\cite{Suguro_2022}. Notice that such bounds are often (see for instance~\cite{Kurokiba_2016,MR1321045}, and earlier references therein:~\cite{MR2136604,Derriennic_1986,Lasota_1994}) referred to as \emph{Shannon's inequality, i.e.,}
\be{Shannon}
-\ird{f\,\log f}\le\frac d2\,\nrm f1\,\log\(\frac{2\,\pi\,\mathsf e}d\,\frac{\ird{|x-\bar x|^2\,f(x)}}{\nrm f1^{1+2/d}}\)
\ee
as stated in~\cite[Lemma~2.3]{Kurokiba_2016}, where $\bar x:=\ird{x\,f(x)}/\nrm f1$ is the center of mass. A classical reference is~\cite{MR32134}, which is (for the concerned part) a reprint of~\cite{MR26286,MR26286b}. However, this work is written in the language of the early days of information theory and it is not easy to rephrase it as above. Notice that the (convex) mathematical entropy has to be understood as a \emph{neg-entropy} (i.e, the physical entropy up to a sign).
Ineq.~\eqref{Shannon} follows from Jensen's inequality $\ird{f\,\log(f/g)}\ge\nrm f1\,\log\nrm f1$ applied with
\[
g(x)=(\sqrt{2\pi})^{-d}\lambda^{-d} \,\mathsf e^{-\,\frac{|x-\bar x|^2}{2\lambda^2}}\,,\quad\lambda^2=\frac{1}d\,\frac{\ird{|x-\bar x|^2\,f(x)}}{\nrm f1}\,.
\]
\smallskip\begin{proposition}\label{Prop:2} If $u\in\mathrm L^2(\R^d)$ is smooth, compactly supported with $\nrm u2=1$ such that $\ird{|u|^2\,\log|u|^2}$ and $\ird{|x|^2\,|u|^2}$ are finite, then \hbox{$|u|^2\,\log|u|^2\in\mathrm L^1(\R^d)$~and}
\[
\ird{\big|\,|u|^2\,\log|u|^2\big|}\le\ird{|u|^2\,\log|u|^2}+\frac d2\,\log\(\frac{2\,\pi}d\ird{|x|^2\,|u(x)|^2}\)+\frac{d+1}{\mathrm e}\,.
\]
\end{proposition}
\begin{proof} Let $f=|u|^2$ if $|u|\le1$, $f\equiv0$ otherwise. Since $\nrm f1\le\nrm u2^2=1$, we have
$-\,\nrm f1\,\log\nrm f1^{1+2/d}\le(d+2)/(d\,\mathsf e)$. Applying~\eqref{Shannon} to
\[
\ird{\big|\,|u|^2\,\log|u|^2\big|}=\ird{|u|^2\,\log|u|^2}-\ird{f\,\log f}
\]
completes the proof using $\ird{|x-\bar x|^2\,f(x)}\le\ird{|x|^2\,f(x)}$.
\end{proof}

If $v$ is a smooth and compactly supported function, we already observed in Section~\ref{Sec:Introduction} that~\eqref{LSI-G} written for $v$ is equivalent to~\eqref{LSI-E} written for $u=\sqrt\gamma\,v$. However, using an integration by parts, we can notice that
\be{Id:G-E}
\nrmg{\nabla v}2^2=\ird{\left|\nabla u+\frac x2\,u\right|^2}=\nrm{\nabla u}2^2+\frac14\ird{|x|^2\,|u|^2}-\frac d2\,\nrm u2^2
\ee
involves a second moment in $x$ of $|u|^2$. It follows that, for a function $v$ in $\mathrm H^1(\R^d,d\gamma)$, the second moment of $|v|^2$ is always finite:
\be{Control-Second-Moment}
\nrmg{\nabla v}2^2+\frac d2\,\nrmg v2^2 \ge \frac 14\irdg{|x|^2\,|v|^2}\,.
\ee
This fact was already observed in~\cite[Ineq.~(4)]{Dolbeault_2012}: combined with the Gaussian Poincar\'e inequality, it shows for instance that
\[
\irdg{|x|^2\,|v|^2}\le2\,(d+1)\irdg{|\nabla v|^2}\quad\forall\,v\in\mathrm H^1(\R^d,d\gamma)\quad\mbox{such that}\quad\irdg v=0\,.
\]
Inequality~\eqref{Control-Second-Moment} cannot hold for an arbitrary function $u \in\mathrm{H}^1(\R^d, dx)$, as the example of $u_n$ given by~\eqref{Example} shows: in that case, we have
\[
\ird{|x|^2\,|u_n|^2}\ge\frac1n\,\sum_{k=0}^{n-1}k^2\,\nrm u2^2\to+\infty\quad\mbox{as}\quad n\to+\infty\,.
\]
By taking~\eqref{LSI-E} into account, we deduce that
\[
\ird{\big|\,|u|^2\,\log|u|^2\big|}\le2\,\nrm{\nabla u}2^2+\frac d2\,\log\(\frac{2\,\pi}d\ird{|x|^2\,|u(x)|^2}\)+\frac{d+1}{\mathrm e}
\]
for any $u\in\mathrm H^1(\R^d,dx)$ such that $\ird{|x|^2\,|u|^2}$ is finite and $\nrm u2=1$.

In the Gaussian framework, we have a similar result as in Proposition~\ref{Prop:2} using $f$ as in the proof and $\irdg{f\,\log f}\ge\nrmg f1\,\log\nrmg f1\ge-1/\mathsf e$.
\smallskip\begin{corollary}\label{Cor:L1Gaussian} If $v\in\mathrm H^1(\R^d,d\gamma)$ is such that $\nrmg v2=1$, then $|v|^2\,\log|v|^2$ is in $\mathrm L^1(\R^d,d\gamma)$ and we have
\[
\irdg{\big|\,|v|^2\,\log|v|^2\big|}\le2\,\nrmg{\nabla v}2^2+\frac2{\mathrm e}\,.
\]
\end{corollary}

\subsection{Improved inequalities under second moment conditions}\label{sss:iiusm}

Here we aim at a \emph{logarithmic uncertainty principle} introduced in~\cite[Lemma~2]{MR3493423} and recently considered in~\cite[Proposition~1.1]{Suguro_2025}, and related stability results with explicit constants. Based on ideas of~\cite[Th.~1.1]{MR3269872},~\cite{MR2294794,zbMATH06709767} and~\cite[Proposition~1]{MR3493423}, the following result holds.
\smallskip\begin{lemma}\label{Prop:DT2016} Let $d\ge1$. With $\varphi$ defined by
\be{Eqn:varphi}
\varphi(t):=\frac d4\(\exp\(\frac{2\,t}d\)-1-\frac{2\,t}d\)\quad\forall\,t\in\R\,,
\ee
we have
\be{Ineq:LogSobGaussianImproved}
\irdg{|\nabla v|^2}-\frac12\irdg{|v|^2\,\log|v|^2}\ge\varphi\(\irdg{|v|^2\,\log|v|^2}+\frac d2-\frac12\irdg{|x|^2\,|v|^2}\)
\ee
for any $v\in\mathrm H^1(\R^d,d\gamma)$ such that $\nrmg v2=1$.\end{lemma}
\begin{proof}[Proof of Lemma~\ref{Prop:DT2016}] Let us give a short proof based on~\cite{MR3493423}. By~\eqref{Control-Second-Moment}, we know that $\irdg{|x|^2\,|v|^2}<\infty$. If $u=v\,\sqrt\gamma$ is such that $\nrm u2=\nrmg v2=1$, the deficit associated to~\eqref{LSI-G} is such that
\begin{align*}
&\hspace*{-0.5cm}\irdg{|\nabla v|^2}-\frac12\irdg{|v|^2\,\log|v|^2}\\
&=\ird{|\nabla u|^2}+\frac14\ird{|x|^2\,|u|^2}-\frac d2\\
&\hspace*{2cm}-\,\frac12\(\ird{|u|^2\,\log|u|^2}+\frac d2\,\log\big(2\,\pi\big)+\frac12\ird{|x|^2\,|u|^2}\)\\
&\ge\frac{\pi\,d\,\mathsf e}2\,\exp\(\frac2d\ird{|u|^2\,\log|u|^2}\)-\,\frac12\ird{|u|^2\,\log|u|^2}-\frac d4\,\log\big(2\,\pi\,\mathrm e^2\big)
\end{align*}
using~\eqref{Id:G-E},
\[
\irdg{|v|^2\,\log|v|^2}=\ird{|u|^2\,\log|u|^2}+\frac d2\,\log\big(2\,\pi\big)+\frac12\ird{|x|^2\,|u|^2}\,,
\]
and~\eqref{LSI-S}. With
\[
t:=\irdg{|v|^2\,\log|v|^2}+\frac d2-\frac12\irdg{|x|^2\,|v|^2}\,,
\]
we have
\[
\ird{|u|^2\,\log|u|^2}=t-\frac d2\,\log\big(2\,\pi\,\mathsf e\big)
\]
which concludes the proof of~\eqref{Ineq:LogSobGaussianImproved}.
\end{proof}

Lemma~\ref{Prop:DT2016} can be rewritten in the Euclidean space via the change of variables
\[
v(x)=\frac{\lambda^{d/4}}{\sqrt{\gamma(x)}}\,u\(\sqrt\lambda\,x\)\quad\forall\,x\in\R^d
\]
as
\begin{multline*}
\ird{|\nabla u|^2}-\frac1{2\,\lambda}\ird{|u|^2\,\log|u|^2}-\frac d{4\,\lambda}\,\log\big(2\,\pi\,\mathrm e^2\,\lambda\big)\\
\ge\frac1\lambda\,\varphi\(\ird{|u|^2\,\log|u|^2}+\frac d2\,\log\(2\,\pi\,\mathrm e\,\lambda\)\)\,.
\end{multline*}
With the choice
\[
\lambda=\frac1d\ird{|x|^2\,|u|^2}\,,
\]
if $\nrm u2=1$, we obtain a stability result for the \emph{logarithmic uncertainty principle}
\begin{multline*}
\ird{|\nabla u|^2}\ird{|x|^2\,|u|^2}-\frac d2\ird{|u|^2\,\log|u|^2}-\frac{d^2}4\,\log\(\frac{2\,\pi\,\mathrm e^2}d\ird{|x|^2\,|u|^2}\)\\
\ge d\,\varphi\(\ird{|u|^2\,\log|u|^2}+\frac d2\,\log\(\frac{2\,\pi\,\mathrm e}d\ird{|x|^2\,|u|^2}\)\)
\end{multline*}
which can be found in~\cite[Proposition~1.1]{Suguro_2025}. This inequality is invariant under scaling. The uncertainty principle, \emph{i.e.}, the fact that the left-hand side is non-negative, is remarkable as it can be seen as an improvement of the standard Heisenberg uncertainty principle, whose stability has been studied in~\cite{MR4255507}. The right-hand side deserves further attention.

For any function $u\in\mathrm L^2(\R^d,dx)$, let us define the \emph{relative entropy} with respect to the positive function $g\in\mathrm L^1(\R^d)$ by
\[
\mathcal{E}[u|g]:=\ird{|u|^2\,\log\(\frac{|u|^2}g\)}-\nrm u2^2+\nrm g1
\]
and consider the set of all positive Gaussian functions
\[
\mathcal M_+:=\left\{g\in\mathrm L^1(\R^d)\,:\,g(x):=\frac M{(2\,\pi\,\lambda)^{d/2}}\,\mathsf e^{-\,\frac{|x-y|^2}{2\,\lambda}}\;\forall\,x\in\R^d\right\}
\]
parametrized by $(M,y,\lambda)\in(0,+\infty)\times\R^d\times(0,+\infty)$. The \emph{best matching} Gaussian function in $\mathcal M_+$ is the function that minimizes $g\mapsto\mathcal{E}[u|g]$ for a given function $u$ and it is determined by the following elementary observation.
\smallskip\begin{lemma}\label{Lem:RelativeEntropy} For any function $u\in\mathrm L^2\big(\R^d,(1+|x|^2)\,dx\big)$ such that $|u|^2\,\log|u|^2\in\mathrm L^1(\R^d)$, we have
\[
\min_{g\in\mathcal M_+}\mathcal{E}[u|g]=\mathcal{E}[u|g_u]
\]
where $g_u$ is the Gaussian function with parameters $(M,y,\lambda)$ such that
\be{myl}
M:=\nrm u2^2\,,\quad y=\frac1M\ird{x\,|u(x)|^2}\quad\mbox{and}\quad\lambda=\frac1{d\,M}\ird{|x-y|^2\,|u(x)|^2}
\ee
and, if $\nrm u2=1$,
\[
\mathcal{E}[u|g_u]=\ird{|u|^2\,\log|u|^2}+\frac d2\,\log\(\frac{2\,\pi\,\mathrm e}d\ird{|x-y|^2\,|u|^2}\)\,.
\]
\end{lemma}
\begin{proof} By convexity, $\mathcal{E}[u|g]$ is non-negative and $\mathcal{E}[u|g]=0$ if and only if $|u|^2=g$. If~$g$ is a Gaussian function with parameters $(M,y,\lambda)$, then
\begin{multline*}
\mathcal{E}[u|g]=\ird{|u|^2\,\log|u|^2}-\nrm u2^2+M-\nrm u2^2\,\log M\\
+\frac d2\,\log(2\,\pi\,\lambda)\,\nrm u2^2+\frac1{2\,\lambda}\ird{|x-y|^2\,|u(x)|^2}
\end{multline*}
and an optimization with respect to $(M,y,\lambda)$ shows the result because the above right-hand side diverges if $M\to0_+$, $M\to+\infty$, $\lambda\to0_+$, $\lambda\to+\infty$, or \hbox{$|y|\to+\infty$}.
\end{proof}

From Lemma~\ref{Prop:DT2016} and Lemma~\ref{Lem:RelativeEntropy}, we recover the result of~\cite[Proposition~1.1]{Suguro_2025}, with the additional property that the choice~\eqref{myl} of the parameters of the Gaussian $g_u$ minimizes the distance to $\mathcal M_+$, measured in terms of the relative entropy, as follows
\[
\irdg{|\nabla v|^2}-\frac12\irdg{|v|^2\,\log|v|^2}\ge\min_{g\in\mathcal M_+}\varphi\(\mathcal{E}[u|g]\)=\mathcal{E}[u|g_u]
\]
for any $u=v\,\sqrt\gamma\in\mathrm H^1(\R^d,dx)$ such that $\nrm u2=1$ and $\ird{x\,|u|^2}=0$.

\medskip Coming back to Inequality~\eqref{Ineq:LogSobGaussianImproved}, we have the following result on the deficit~\eqref{delta}, which was already known from~\cite[Theorem~1, 1.]{Indrei_2023} with a different proof.
\smallskip\begin{corollary}\label{Cor:Moments} Let $d\ge1$. Let us consider a sequence $(v_n)_{n\in\N}$ of functions in $\mathrm H^1(\R^d,d\gamma)$ such that $\nrmg{v_n}2=1$ for any $n\in\N$. If $\limsup_{n\to+\infty}\irdg{|x|^2\,|v_n|^2}\le d$, then $\lim_{n\to+\infty}\delta[v_n]=0$ is equivalent to the convergence of $(v_n)_{n\in\N}$ to $1$ in $\mathrm H^1(\R^d,d\gamma)$, and then we have $\lim_{n\to+\infty}\irdg{|x|^2\,|v_n|^2}=d$. \end{corollary}
\smallskip With minimal effort, as a consequence of~\eqref{Ineq:LogSobGaussianImproved}, we can also recover the full statement of~\cite[Theorem~1]{Indrei_2023} which asserts that, for any sequence $(v_n)_{n\in\N}$, such that $\lim_{n\to+\infty}\delta[v_n]=0$, 
the two following properties are equivalent:\\
(i) $v_n \to 1$ in $\mathrm H^1(\mathbb R^d,d\gamma)$ as $n \to +\infty$,\\
(ii) $\lim_{n\to+\infty}\irdg{|x|^2\,|v_n|^2}=d$.

\medskip With $\varphi$ defined by~\eqref{Eqn:varphi}, we may notice that $\varphi''(t)=(1/d)\,\exp(2\,t/d)$ for any $t\in\R$ and, as a consequence, $\varphi''(t)\ge1/d$ if $t\ge0$. Thus,
\be{Stab:0}
\delta[v]\ge\frac1{2\,d}\(\irdg{|v|^2\,\log|v|^2}\)^2+\frac1{8\,d}\(d-\irdg{|x|^2\,|v|^2}\)^2
\ee
for any $v\in\mathrm H^1(\R^d,d\gamma)$ such that $\nrmg v2=1$ and \hbox{$\irdg{|x|^2\,|v|^2}\le d$}.
Corollary~\ref{Cor:Moments} is a consequence of~\eqref{Stab:0} under the condition that $\limsup_{n\to+\infty}\irdg{|x|^2\,|v_n|^2}\le d$. Indeed, by using the {Csisz\'ar-Kullback-Pinsker inequality}
\[
\irdg{|v|^2\,\log|v|^2}\ge\frac14\(\irdg{\big|\,|v|^2-1\big|}\)^2
\]
for any $v\in\mathrm H^1(\R^d,d\gamma)$ such that $\nrmg v2=1$, see~\cite{Csiszar1967,Kullback1967,MR0213190}, and the Brezis-Lieb lemma, see~\cite[Theorem~2]{zbMATH03834677}, one can then prove that the above sequence $(v_n)_{n\in\N}$ converges to $1$ in $\mathrm H^1(\R^d,d\gamma)$. See~\cite{Indrei_2023} for further details. In fact, one can directly obtain an explicit stability estimate from~\eqref{Stab:0}, which goes as follows.
\smallskip\begin{corollary}\label{Cor:Stab} Let $d\ge1$. For any $v\in\mathrm H^1(\R^d,d\gamma)$ such that $\nrmg v2=1$ and $\irdg{|x|^2\,|v|^2}\le d$, we have
\[
\irdg{|\nabla v|^2}-\frac12\irdg{|v|^2\,\log|v|^2}\ge\frac d4\,\chi\(\frac4d\irdg{|\nabla v|^2}\)
\]
with $\chi(s):=1+s-\sqrt{1+2\,s}$. \end{corollary}
\begin{proof} With $\mathsf e:=\irdg{|v|^2\,\log|v|^2}$ and $\mathsf i:=\irdg{|\nabla v|^2}$, from~\eqref{Stab:0} we obtain
\[
\mathsf e^2+d\,\mathsf e-2\,d\,\mathsf i\le0
\]
which can be inverted as $\mathsf e\le\big(\sqrt{d\,(d+8\,\mathsf i)}-d\big)/2$ and shows that
\[
\mathsf i-\frac12\,\mathsf e\ge\frac14\(4\,\mathsf i+d-\sqrt{d\,(d+8\,\mathsf i)}\)\,.
\]
This completes the proof of Corollary~\ref{Cor:Stab}.
\end{proof}

In fact, under the assumptions of Corollary~\ref{Cor:Stab}, instead of using~\eqref{Stab:0}, we can rewrite~\eqref{Ineq:LogSobGaussianImproved} as
\[
\mathsf i\ge\frac{\mathsf e}2+\varphi(\mathsf e)=\psi^{-1}\(\frac{\mathsf e}2\)
\]
where the last equality defines the concave increasing function $\psi$ such that
\[
\psi(s):=\frac d2\,\log\(1+\frac{4\,s}d\)\quad\forall\,s\ge0
\]
and obtain
\[
\irdg{|\nabla v|^2}-\frac12\irdg{|v|^2\,\log|v|^2}\ge\xi\(\irdg{|\nabla v|^2}\)
\]
for any $v\in\mathrm H^1(\R^d,d\gamma)$ such that $\nrmg v2=1$ and $\irdg{|x|^2\,|v|^2}\le d$. This is an improvement of Corollary~\ref{Cor:Stab} because
\be{Eqn:xi}
\xi(s):=s-\frac12\,\psi(s)\ge\frac d4\,\chi\(\frac{4\,s}d\)\quad\forall\,s>0\,,
\ee
so that $\xi\(\irdg{|\nabla v|^2}\)=0$ if and only if $v$ is constant. Notice that $\xi$ is convex and positive.

\medskip As a consequence, in Euclidean variables, we also obtain a stability result for the \emph{logarithmic uncertainty principle} in the strongest possible norm, that goes as follows. If we define the \emph{relative Fisher information} with respect to the positive function $g\in\mathrm L^1(\R^d)$ by
\[
\mathcal I[u|g]:=\ird{\left|\nabla\(\frac u{\sqrt g}\)\right|^2\,g}
\]
and, with $y$ and $\lambda$ given by~\eqref{myl}, rewrite Lemma~\ref{Prop:DT2016} in the Euclidean space via the change of variables
\[
v(x)=\frac{\lambda^{d/4}}{\sqrt{\gamma(x)}}\,u\(\sqrt\lambda\,(x+y)\)
\]
then we have
\[
\lambda\,\mathcal I[u|g_u]-\frac 12\,\mathcal E[u|g_u]\ge\varphi\big(\mathcal E[u|g_u]\big)
\]
\smallskip\begin{corollary}\label{Cor:Stab2} Let $d\ge1$. For any $u\in\mathrm H^1(\R^d)$ such that $\nrm u2=1$ and $\ird{|x|^2\,|u|^2}$ is finite, with $y$ and $\lambda$ given by~\eqref{myl}, we have
\[
\ird{|\nabla u|^2}\ird{|x-y|^2\,|u|^2}-\frac d2\ird{|u|^2\,\log|u|^2}-\frac{d^2}4\,\log\(2\,\pi\,\mathrm e^2\lambda\)=:\delta[u]
\]
where $\delta[u]\ge d\,\varphi\big(\mathcal E[u|g_u]\big)$ and $\delta[u]\ge d\,\xi\big(\mathcal I[u|g_u]\big)$ with $\varphi$ and $\xi$ given respectively by~\eqref{Eqn:varphi} and~\eqref{Eqn:xi}.
\end{corollary}
\smallskip\noindent The proof is a simple rewriting of the previous computations. In Corollary~\ref{Cor:Stab2}, the inequalities provide upper estimates of the distances to $\mathcal M_+$ using, for instance, the Csisz\'ar-Kullback-Pinsker inequality.

\newpage\section{Stability}\label{Sec:Stability}

\subsection{Optimal constants and optimal functions}\label{Sec:Opt}

Inequalities~\eqref{LSI-G},~\eqref{LSI-E},~\eqref{LSI-E-lambda} and~\eqref{LSI-S} can be rewritten for functions $u\in\mathrm H^1(\R^d,dx)$ and \hbox{$v\in\mathrm H^1(\R^d,d\gamma)$} respectively as
\begin{subequations}
\begin{align}
&\label{LSI-Gb}\textstyle
\irdg{|\nabla v|^2}\ge\frac12\irdg{|v|^2\,\log\(\frac{|v|^2}{\nrmg v2^2}\)}\,,\\
&\label{LSI-Eb}\textstyle
\ird{|\nabla u|^2}\ge\frac12\ird{|u|^2\,\log\(\frac{|u|^2}{\nrm u2^2}\)}+\frac d4\,\log\big(2\,\pi\,\mathrm e^2\big)\,\nrm u2^2\,,\\
&\label{LSI-E-lambdab}\textstyle
\ird{|\nabla u|^2}\ge\frac1{2\,\lambda}\ird{|u|^2\,\log\(\frac{|u|^2}{\nrm u2^2}\)}+\frac d{4\,\lambda}\,\log\big(2\,\pi\,\mathrm e^2\,\lambda\big)\,\nrm u2^2\,,\\
&\label{LSI-Sb}\textstyle
\ird{|\nabla u|^2}\ge\frac{\pi\,d\,\mathsf e}2\,\nrm u2^2\,\exp\(\frac2d\ird{\frac{|u|^2}{\nrm u2^2}\,\log\(\frac{|u|^2}{\nrm u2^2}\)}\)\,,
\end{align}
\end{subequations}
without any normalization in either $\mathrm L^2(\R^d,dx)$ nor $\mathrm L^2(\R^d,d\gamma)$. These inequalities are written with optimal constants as can be checked using $v_\varepsilon(x)=1+\varepsilon\,x\cdot\nu$ in the limit as $\varepsilon\to0$ for any given $\nu\in\mathbb S^{d-1}$ in case of~\eqref{LSI-G}, $u=\sqrt\gamma$ in case of~\eqref{LSI-E} and~\eqref{LSI-S}, and $u=\lambda^{d/4}\,\gamma^{1/2}(\cdot/\sqrt\lambda)$ in case of~\eqref{LSI-E-lambda}. The next issue is to identify all optimal functions. The first explicit result for~\eqref{LSI-G} is due to E.~Carlen in~\cite{MR1132315}, although the \emph{carr\'e du champ} method of D.~Bakry and M.~Emery in~\cite{MR889476} applies: we refer to~\cite{MR4773135} for more detailed explanations. Since~\eqref{LSI-Gb},~\eqref{LSI-Eb},~\eqref{LSI-E-lambdab} and~\eqref{LSI-Sb} are equivalent for smooth and sufficiently decreasing functions as explained in Section~\ref{Sec:Introduction}, cases of equality can be reduced to optimality for any of these inequalities.
\smallskip\begin{theorem}\label{thm:opti}~
\begin{enumerate}
\item[\rm 1)] A function $v$ is optimal in~\eqref{LSI-Gb} if and only if $v(x)=v_{a,b}(x) :=a\,\mathrm e^{b\cdot x}$, for any $a \in \mathbb R$ and $b \in \R^d$.
\item[\rm 2)] A function $u$ is optimal in~\eqref{LSI-Eb} if and only if $u(x)=u_{a,b}(x) :=a\,\mathrm e^{-|x-b|^2/2}$, for any $a \in \mathbb R$ and $b \in \R^d$.

\item[\rm 3)] For any fixed $\lambda >0$, a function $u$ is optimal in~\eqref{LSI-E-lambdab} if and only if $u(x)=u_{a,b,\lambda}(x) :=a\,\mathrm e^{-|x-b|^2/(2\,\lambda)}$, for any $a \in \mathbb R$ and $b \in \R^d$.
\item[\rm 4)] A function $u$ is optimal in~\eqref{LSI-Sb} if and only if $u(x)=u_{a,b,\lambda}(x)=a\,\mathrm e^{-|x-b|^2/(2\,\lambda)}$, for any $a \in \mathbb R$, $b \in \R^d$, and $\lambda>0$.
\end{enumerate}
\end{theorem}
\noindent Cases 1) and 2) were explicitly established by E.~Carlen in~\cite[Theorem~4]{MR1132315}. Alternatively, we give a proof based on the \emph{carr\'e du champ} method of~\cite{MR889476}, which directly shows Case 4) and has been used in this context only in~\cite{zbMATH01503413}. Here we use the pressure variable in the computations, as for instance in~\cite{MR4773135}.
\begin{proof} Let us give a sketch of a proof based on the \emph{R\'enyi entropy power} computation. Here we work at formal level and refer to~\cite{zbMATH01503413} for the origin of this method. Assume that $\rho=|u|^2=\mathrm e^\PP$ solves the heat equation
\be{heat}
\frac{\partial\rho}{\partial t}=\Delta\rho=\nabla\cdot(\rho\,\nabla\PP)
\ee
so that the \emph{pressure} variable $\PP=\log\rho$ and $u>0$ respectively solve
\[
\frac{\partial\PP}{\partial t}=\Delta\PP+|\nabla\PP|^2\quad\mbox{and}\quad\frac{\partial u}{\partial t}=\Delta u+\frac{|\nabla u|^2}u\,.
\]
Further assuming that the function $u$ is smooth and rapidly decaying as $|x|\to+\infty$, a straightforward computation shows that the entropy decays according to
\[
\frac d{dt}\ird{\rho\,\log\rho}=-\ird{\rho\,|\nabla\PP|^2}=-\,4\ird{|\nabla u|^2}
\]
while the Fisher information obeys to
\begin{align*}
\frac d{dt}\ird{|\nabla u|^2}=&\,-\,2\ird{\Delta u\(\Delta u+\frac{|\nabla u|^2}u\)}\\
=&\,-\,2\ird{\(\|\HH u\|^2-2\,\HH u:\frac{\nabla u\otimes\nabla u}u+\frac{\|\nabla u\otimes\nabla u\|^2}{u^2}\)}
\end{align*}
where $A:B=\sum_{i,j=1}^da_{ij}\,b_{ij}$ denotes the standard contraction of matrices $A$ and $B$ and $\|A\|^2=A:A$. Using $\PP=2\,\log u$, $u\,\nabla \PP=2\,\nabla u$,
\[
\frac{\nabla u\otimes\nabla u}{u^2}=\frac14\,\nabla\PP\otimes\nabla\PP\quad\mbox{and}\quad\HH u=\frac u2\(\HH\PP+\frac12\,\nabla\PP\otimes\nabla\PP\)\,,
\]
we conclude that
\[
\frac d{dt}\ird{|\nabla u|^2}=-\,\frac12\ird{\rho\,\|\HH\PP\|^2}\,.
\]
By conservation of mass, we can assume that $\nrm{\rho(t,\cdot)}1=\nrm u2^2=1$ for any $t\ge0$ if $\rho$ solves~\eqref{heat}, so that
\begin{multline*}
\exp\(\frac 2d\ird{\rho\,\log\rho}\)\frac d{dt}\(\ird{|\nabla u|^2}\,\exp\(-\frac 2d\ird{\rho\,\log\rho}\)\)\\
=-\frac 12\ird{\rho\,\|\HH\PP\|^2}+\frac1{2\,d}\(\ird{\rho\,|\nabla\PP|^2}\)^2\\
=-\frac 12\ird{\rho\,\left\|\HH\PP-\frac{1}{ d}\ird{\rho\,|\nabla\PP|^2}\,\mathrm{Id}\right\|^2}\,.
\end{multline*}
These computations can be justified using approximations based on smooth and compactly supported initial data $\rho(0,\cdot)$, for which $\rho(t,\cdot)$ is uniformly-log-concave, thus fast-decaying at infinity as well as its derivatives (as they are also solving the heat equation with smooth and compactly-supported initial data). Now let us consider an optimizer $u$, which can be taken positive without loss of generality, and take it as an initial datum in the above computations. We obtain that
\[
\ird{\rho\,\left\|\HH\PP-\frac{1}{ d}\ird{\rho\,|\nabla\PP|^2}\,\mathrm{Id}\right\|^2}=0\,,
\]
\emph{i.e.}, $\PP=2\,\log u=\alpha\,|x-x_0|^2+\beta$ for some constants $\alpha$ and $\beta$ and for some \hbox{$x_0\in\R^d$}. Since~\eqref{LSI-Gb},~\eqref{LSI-Eb},~\eqref{LSI-E-lambdab} and~\eqref{LSI-Sb} share the same optimizers up to obvious trans\-for\-ma\-tions, this completes the sketch of the proof of Theorem~\ref{thm:opti}.
\end{proof}

\subsection{Stability results in the Gaussian setting}

\subsubsection{Improved inequalities}\label{Sec:ImprovedInequalities}

A first improvement of~\eqref{LSI-G} has been formulated in~\cite{MR1132315} by E.~Carlen, in terms of the Wiener transform~$\mathcal{W}$ and the Beckner-Hirschman inequality (see~\cite{MR89127},~\cite[Section~IV.3]{MR385456}, with sharp constant due to I.~Bia\l ynicki-Birula and J.~Mycielski in~\cite{MR386550}, and Beckner in~\cite{MR385456}), which is also known as the \emph{entropic uncertainty principle},
\[
\delta[v] \ge \frac12 \irdg{|\mathcal W v|^2\,\log |\mathcal{W} v|^2}\,,\quad \nrmg v2=1\,.
\]
Under a non-negativity assumption, E.~Carlen proved that the right-hand side in the above inequality is non-negative and vanishes if and only if $v \in \mathcal M$. According to Theorem~\ref{thm:opti}, the set $\mathcal M$ of optimisers is made of the functions $v_{a,b}(x) :=a\,\mathrm e^{b\cdot x}$, for any $a \in \mathbb R$ and $b \in \R^d$ (recall that we are considering real valued functions only). The method relies first on Cram\'er's convolution result in~\cite{MR1545629}, proving the result for non-negative functions. Then, in~\cite[p. 198]{MR1132315}, the technique is extended to signed functions. It points in the direction of the \emph{entropic uncertainty principle}, complex-valued functions and the Fourier transform. See for instance the results of~\cite{MR1254832} and the $\mathrm L^2$ conditional stability result of~\cite[Corollary~4.7]{Feo_2016}.

Another direct improvement of~\eqref{LSI-G} can be obtained using the \emph{carr\'e du champ} method of~\cite{MR889476}, which we sketch briefly. With respect to $\gamma$, let us define the \emph{relative Fisher information} and the \emph{relative entropy} functionals by $\mathcal{I}[v]=\nrmg{\nabla v}2^2$ and $\mathcal{E}[v]=\irdg{|v|^2\,\log|v|^2}$, for $v\in\mathrm H^1(\R^d,d\gamma)$ such that $\nrmg{v}2=1$. Next, assume that $|w|^2$ solves the Ornstein--Uhlenbeck equation so that $w=w(t,x)$ is the solution of
\be{eq:OU}
\frac{\partial w}{\partial t}=\Delta w + \frac{|\nabla w|^2}{w}- x \cdot \nabla w\,, \quad w(t=0,\cdot)=v\,.
\ee
According, for instance, to~\cite[Section~2.2]{MR4773135}, it holds true that
\begin{equation}
\label{eq:BE}
\frac{d}{dt} \mathcal{E}[w(t,\cdot)]=-\,4\,\mathcal I[w(t,\cdot)]\,,\quad \frac{d}{dt} \mathcal{I}[w(t,\cdot)] + 2\,\mathcal{I}[w(t,\cdot)]=-\,2 \irdg{ \left\|\HH\PP\right\|^2\,|w|^2}\,,
\end{equation}
where $\PP=2\,\log w$ is the pressure variable as in Section~\ref{Sec:Opt}.
Integrating on $(0,\infty)$,~\eqref{eq:BE} implies
\begin{equation}
\label{eq:LSIs}
\delta[v] \ge \int_0^\infty \mathcal{R}[w(t,\cdot)]\,dt\quad\mbox{where}\quad\mathcal{R}[w]:=2\irdg{ \left\|\HH\PP\right\|^2\,|w|^2}
\,,
\end{equation}
where $\mathcal{R}$ vanishes if and only if $v $ is an optimiser of~\eqref{LSI-G}.
Additional information can be extracted from $\mathcal{R}$, for some classes of functions~$v$ as we shall see next.

\subsubsection{Functions with asymptotic exponential or Gaussian behaviour}\label{Sec:Exp}

If the measure $|v|^2\,d\gamma$ satisfies the \emph{Poincar\'e inequality}
\begin{equation}\label{eq:cpoi}
\irdg{|\nabla \phi|^2\,|v|^2} \ge C_P\irdg{\left|\phi - \irdg{\phi\,|v|^2}\right|^2\,|v|^2}\quad\forall\,\phi \in \mathrm{C}^\infty_c(\R^d)
\end{equation}
for some positive constant $C_P$ and if $w$ solves~\eqref{eq:OU} with initial datum $w(t=0,\cdot)=v$,
the same holds true for the measure $|w(t,\cdot)|^2\,d\gamma$ for all $t \ge 0$, with $C_P(0)=C_P$ and
\[
C_P(t)=\frac{C_P}{C_P + \big(1-C_P\big)\,\mathrm e^{-2t}}\,.
\]
In addition, if $v$ is centered, \emph{i.e.}, $\irdg{x\,|v|^2}=0$, then $\PP(t,\cdot)=2\,\log w(t,\cdot)$ is such that $\irdg{ \nabla \PP(t,\cdot)\,|w(t,\cdot)|^2 }=0$, and by the above Poincar\'e inequality with constant $C_P(t)$ applied to $\partial\PP/\partial x_i$ for each $i=1$, $2$,\ldots$d$, we obtain
\[
\mathcal{R}\big[w(t,\cdot)\big] \ge C_P(t)\irdg{|\nabla \PP(t,\cdot)|^2\,|w(t,\cdot)|^2}=C_P(t)\,\mathcal{I}[w]\,.
\]
In~\cite{zbMATH06661531}, this argument allows M.~Fathi, E.~Indrei, and M.~Ledoux to prove that
\[
\delta[v] \ge \frac{C_P^2 -C_P-C_P\,\log C_P }{(1-C_P)^2}\irdg{|\nabla v|^2}
\]
for all centered functions $v$ satisfying~\eqref{eq:cpoi}.

The result of~\cite{zbMATH06661531} can be generalised as follows. Let us call $\mathcal V$ the space of centered functions $v$ such that $v$ admits~\eqref{eq:cpoi} for some positive constant $C_P$. The flow~\eqref{eq:OU} preserves $\mathcal V$.
In addition, assume that for some $T \in (0,\infty)$, the solution $w(t,\cdot)$ to~\eqref{eq:OU} with initial datum $v$ belongs to~$\mathcal V$ at $t=T$, hence, for any $t\ge T$.
Then we obtain
\[
\delta[v] \ge \mathrm e^{-2\,T}\,\frac{C_P^2 -C_P-C_P\,\log C_P }{(1-C_P)^2}\irdg{|\nabla v|^2}
\]
using the \emph{backwards-in-time} estimate of~\cite{MR4773135} and the result of~\cite{zbMATH06661531}. The existence of such a finite $T$ is granted if $v$ is a compactly supported function.
In~\cite{MR4316725}, H.-B.~Chen, S.~Chewi, and J.~Niles-Weed provide a more general sufficient condition: if for some $\varepsilon>0$ and $\mathcal{C}>0$,
\be{CCBB}
\iint_{\R^d\times\R^d} |v(x)|^2\,|v(y)|^2\,\mathrm e^{\varepsilon|x-y|^2}\,\gamma(x)\,dx\,\gamma(y)\,dy \le \mathcal{C}\,,
\ee
then the solution $w(t,\cdot)$ to~\eqref{eq:OU} has the property for an explicit $T>0$ depending on~$\varepsilon$ and $\mathcal{C}$ but not on the dimension $d$. Note that the Gaussian-tail condition~\eqref{CCBB} cannot be created along the flow~\eqref{eq:OU}, see~\cite{Herraiz1999}, without additional assumptions. As a result, proved in~\cite{MR4773135}, there is an explicit constant $\mathsf{c}=\mathsf{c}(\varepsilon, \mathcal{C})$ such that
\[
\delta[v] \ge \mathsf{c}\irdg{|\nabla v|^2}
\]
under Condition~\eqref{CCBB}.

It is currently an open question to decide whether $T$ is finite for a function $v\in\mathrm H^1(\R^d,d\gamma)$ under the more natural assumption $\irdg{|v|^2\,\mathrm e^{\theta\,|x|}}<\infty$ for some $\theta>0$, which is weaker than Condition~\eqref{CCBB}. For functions with a finite exponential moment, there are stability results based on a weaker notion of distance. See~\cite{MR3271181},~\cite{MR3269872} and~\cite[ineq.~(33)]{Bolley_2018}. If $|v|^2$ can be written in the form $|v|^2=\mathrm e^{-h}\,d\gamma$ for $h$ such that $-1+\varepsilon\le \HH h \le M$ for some $\varepsilon>0$, then
\[
\delta[v] \ge \beta(\varepsilon,M)\,\mathrm{W}_2^2(|v|^2\,dx, \gamma)
\]
where $\mathrm{W}_2$ is the $2$-Wasserstein distance, see~\cite[Theorem~1.1]{MR3271181}.
For a more recent insight upon the relation between log-Hessian bounds, the Ornstein--Uhlenbeck flow, and the stability of~\eqref{LSI-G}, we refer to~\cite{brigati2024heat}.

Finally, we notice that all results in this section are optimal with respect to the exponent of the distance, which is sometimes referred in the literature as \emph{sharp qualitative stability}.

\subsubsection{Functions with finite second moment}
Another possible way to exploit the improvement~\eqref{eq:LSIs} is described below, for functions~$v$ such that $\irdg{x\,|v|^2}=0$ and $\irdg{|x|^2\,|v|^2} \le d$. The resulting estimate has been written in~\cite{MR3493423} using the comparison of~\eqref{LSI-G} and~\eqref{LSI-S}, when the second moment is exactly \hbox{$\irdg{|x|^2\,|v|^2}=d$}. Otherwise, we attribute the result and the corresponding proof to~\cite{MR3269872}, even though the key-estimate appears in~\cite{zbMATH01503413} as well.

Going back to~\eqref{eq:LSIs}, using the Cauchy--Schwarz and the arithmetic-geometric inequalities as in~\cite[proof of Lemma~3]{MR4773135}, we can write
\[
\mathcal{R}[w(t,\cdot)] \ge \frac1{d} \left( \irdg{(\Delta \PP)\,|w|^2} \right)^2 \ge \frac{4}{d}\,\left( \irdg{|\nabla w|^2} \right)^2,
\]
where the last estimate is achieved using the condition on the second moment. By solving the differential inequality obtained from~\eqref{eq:BE} for $t\in\R^+$, we find
\[
\delta[v]\,\ge \xi \left( \irdg{|\nabla v|^2}\right)\,,\quad\mbox{where}\quad \xi(s) :=s - \frac{d}{4}\,\log \left(1 + \frac{4}{d}\,s\right)\quad\forall\,s>0
\]
is defined as in~\eqref{Eqn:xi}. This provides an alternative proof to the results of Section~\ref{sss:iiusm}.

For $s \to 0$, we notice that $\xi(s)=2\,s^2/d + o(s^2)$, which means that the extra term we found is of the order of $\nrmg{\nabla v}2^4$ for $\nrmg{\nabla v}2$ small, as in Corollary~\ref{Cor:Stab}. In Section~\ref{Sec:Exp}, we found a remainder term of order $2$. Identifying the minimal conditions for the existence of a positive constant $\beta$ such that $\delta[v] \ge \beta \irdg{|\nabla v|^2}$, for centered functions $v$ with $\irdg{|x|^2\,|v|^2} \le d$, is however still an open question.

\medskip As discussed in Section~\ref{sss:iiusm}, the condition $\irdg{|x|^2\,|v|^2} \le d$ is sufficient for local stability results for~\eqref{LSI-G} around constant functions. This is also true in weaker distances such as as~$\mathrm{W}_2$.
On the other hand, an improvement of~\eqref{LSI-G} for functions $|v|^2$ with arbitrarily large but finite second order moment holds in two known cases. As found out by E.~Indrei:
\begin{itemize}
\item In~\cite[Theorem~1(2)]{Indrei_2023}, it is shown that, for all $b>0$, there exists a constant $\beta>0$, such that, for all centered functions $v\in\mathrm L^2(\R^d,d\gamma)$ such that $\irdg{|x|^4\,|v|^2 } \le b$,
\[
\||v|-1\|_{\mathrm H^1(\R^d,d\gamma)}^2 \le \beta\,\left( \delta[v] + \delta[v]^{1/2} \right)\,.
\]

\item Stability in $\mathrm W^{1,1}(\R,d\gamma)$ is proved in~\cite[Theorem~1.1]{Indrei_2025}, in dimension $d=1$. For all $a>0$, there exist $\alpha >0$ such that for all non-negative, normalized and centered functions $v\in\mathrm H^1(\R,d\gamma)$ with $\irdg{|x|^2\,|v|^2}\leq a$, it holds
\[
\big\||v|^2-1\big\|_{\mathrm{W}^{1,1}(\R,d\gamma)} \le \alpha\,\left( \delta^{1/4}[v] + \delta^{3/4}[v] \right).
\]
\end{itemize}
Whether the exponents in these last two results are optimal and how they can be extended to $d>1$ in the second case, according to~\cite{Indrei_2025}, are open questions.

\subsubsection{Stability in \texorpdfstring{$\mathrm L^2$}{L2} without moment bounds}

We refer to~\cite{MR4455233,MR4305006} for a review on stability results in $\mathrm L^p$-norms, which still leaves some open questions like, for instance, the question of optimal exponents in the stability estimates.
Stability in $\mathrm L^2$-norm was an open problem until recently.
In~\cite{DEFFL2024,DEFFL2025}, J.~Dolbeault, M.~Esteban, A.~Figalli, R.~Frank, and M.~Loss construct an explicit, positive, dimension-free constant $\beta$ such that
\begin{equation}\label{deffl}
\forall\,v \in \mathrm H^1(\R^d,d\gamma)\,,\quad \delta[v] \ge\,\beta\,\inf_{v_{a,b}\in \mathcal M}\,\|v-v_{a,b}\|^2_{\mathrm L^2(\R^d,d\gamma)}\,,
\end{equation}
where $\mathcal M$ and $v_{a,b}$ are defined in Section~\ref{Sec:ImprovedInequalities}.
The exponent in the right-hand side of~\eqref{deffl} is optimal (see for instance~\cite[Theorem~2]{Indrei_2023}) for homogeneity reasons. In~\cite[Theorem~1]{Indrei_2023}, the author also studies the stability in $\mathrm H^1(\R^d,d\gamma)$ along sequences of functions, depending on their moments.

Even though~\eqref{deffl} can be proved directly (see~\cite{DEFFL2024}), an interesting feature of this estimate is that is can be recovered as a large-dimensional limit of the constructive stability estimate of Sobolev's inequality on the sphere, according to~\cite{DEFFL2025}.
The striking optimality of the constant $1/2$ in~\eqref{LSI-G}, regardless of the topological dimension $d$ of the space, means that~\eqref{LSI-G} can be interpreted as an \emph{infinite-dimensional} inequality in terms of the modern theory of metric measure spaces and synthetic curvature-dimension conditions: we refer to the work of L.~Ambrosio, N.~Gigli, and G.~Savar\'e~\cite{MR3298475} for further details. However, the heuristics that \emph{the Gaussian measure behaves similarly to the unitary measure on a very large-dimensional sphere} is present in mathematics since the XIX$^{th}$ century, at least, and we refer to~\cite{Vershik:2007nr} for a complete historical account.
For completeness, let us review next a few recent results of stability of functional inequalities on the sphere, that are related with~\eqref{LSI-G}.

\subsubsection{Interpolation inequalities on the sphere}

One feature of~\eqref{LSI-G} is the \emph{criticality}, a concept related to maximal embeddings of Orlicz spaces studied for instance by A.~Cianchi and L.~Pick in~\cite{MR2514054}. We specialise this notion to the particular case, of Beckner's \emph{Gaussian interpolation inequalities} introduced in~\cite{MR954373}.
For all $p \in[1,2)$ and all $v \in \mathrm H^1(\mathbb R^d,d\gamma)$, the following inequality holds
\begin{equation}\label{beckner}
\nrmg{\nabla v}2^2 - \frac1{2-p}\(\nrmg v2^2- \nrmg vp^2\)\ge 0\,.
\end{equation}
Inequality~\eqref{LSI-G} represents the \emph{critical} upper endpoint as $p \uparrow 2.$
Note that for $p=1$, we recover the \emph{Gaussian Poincar\'e inequality}.

On the $n$-dimensional unit sphere $\mathbb S^n$, we have a similar family of interpolation inequalities, due to~\cite{MR1134481,MR1213110}, and obtained independently later in~\cite{MR1230930}.
Those are a family of \emph{Gagliardo--Nirenberg--Sobolev inequalities}, defined by a parameter $p \in [1,2)\cup(2,2^\star]$, where $2^\star=2\,n/(n-2)$, for $n \ge 3$, and for any $p\in[1,2)\cup(2,+\infty)$ if $n=1$ or $2$, which interpolates between the Poincar\'e inequality ($p=1$), and the critical Sobolev inequality ($p=2^\star$) if $n\ge3$. Under these conditions, for all $F \in \mathrm H^1(\mathbb S^n,d\mu_n)$, where~$d\mu_n$ denotes the uniform probability measure on $\mathbb S^n$, we have
\begin{equation}
\label{GNSm}
\int_{\mathbb S^n} |\nabla F|^2\,d\mu_n - \frac{d}{p-2}\(\|F\|^2_{\mathrm L^p(\mathbb S^n,d\mu_n)}-\|F\|^2_{\mathrm L^2(\mathbb S^n,d\mu_n)}\) \ge 0\quad\mbox{if}\quad p\neq2
\end{equation}
and for the limit case $p=2$, the (subcritical) \emph{logarithmic Sobolev inequality}
\be{GNS-log}
\int_{\mathbb S^n} |\nabla F|^2\,d\mu_n - \frac2{d}\,\int_{\mathbb S^n} |F|^2\,\log\(\frac{|F|^2}{\|F\|_{\mathrm L^2(\mathbb S^n,d\mu_n)}}\)d\mu_n \ge 0\,.
\ee
Inequality~\eqref{GNSm} can be proved via the entropy method, using nonlinear diffusion flows. The interested reader may refer to~\cite{MR620581,MR3229793,MR1338283,1504}, and~\cite{ji2024boundsoptimalconstantbakryemery,ji2024dissipationestimatesfisherinformation,ji2024entropydissipationestimateslandau}, where further computations for the heat equation and the Fisher information on Riemannian manifolds are also carried out.

It turns out that for all $v\in\mathrm H^1(\R^d,d\gamma)$, the sequence of functions $(F_n)_{n\in\N}$ of functions of $\mathrm H^1(\mathbb S^n,d\mu_n)$ such that
\[
F_n(\omega_1,\omega_2,\ldots\omega_d,\omega_{d+1}\ldots\omega_{n+1})=v\(\omega_1/\sqrt n,\omega_2/\sqrt n,\ldots,\omega_d/\sqrt n\)
\]
satisfies
\begin{multline*}
\lim_{n\to\infty}\left( \int_{\mathbb S^n} |\nabla F_n|^2\,d\mu_n - \frac{d}{p_n-2}\(\|F_n\|^2_{\mathrm L^{p_n}(\mathbb S^n,d\mu_n)}-\|F_n\|^2_{\mathrm L^2(\mathbb S^n,d\mu_n)}\)\)\\
=\nrmg{\nabla v}2^2 - \frac1{2-p_n}\(\nrmg v2^2- \nrmg v{p_n}^2\)
\end{multline*}
if $(p_n)_{n\in\N}$ is a sequence of exponents in $[1,2)\cup(2,2^\star)$ such that $\lim_{n \to \infty}p_n=p\in[1,2)$, see~\cite{MR4703671}. The case $p_n=2^\star=2\,n/(n-2)\to p=2$ as $n\to+\infty$ is covered in~\cite{DEFFL2024,DEFFL2025}. Heuristically, the function $v$ has to be seen as the stereographic projection of a \hbox{$d$-marginal} of $F_n$ for any $n>d$, large enough, if we assume for instance that $v$ is compactly supported. See~\cite{MR4703671} for a detailed statement.

\begin{itemize}
\item For $p=2^\star,$~\eqref{GNSm} is the critical Sobolev inequality on $\mathbb S^n$ and the optimisers are given by the Aubin--Talenti manifold $\mathcal M$ made of the functions $G(x)=c\,(1+ b \cdot x)^{-(n-2)/2}$ such that $c \in \mathbb R$ and $b\in\R^{n+1}$ with $|b|<1$. There is a well known stability result which follows from~\cite{MR1124290} using an inverse stereographic projection and shows that the deficit in~\eqref{GNSm} if $p=2^\star$ is bounded from below, up to a constant, by $\mathsf d^2(F,G):=\inf_{G\in \mathcal M} \big( \|\nabla F - \nabla G\|^2_{\mathrm L^2(\mathbb S^n,d\mu_n)}+\frac d{p-2}\,\|F - G\|^2_{\mathrm L^2(\mathbb S^n,d\mu_n)}\big)$. The main result of~\cite{DEFFL2024} is the fact that the stability constant is bounded from below by $\beta/n$, with $\beta$ as in~\eqref{deffl}, and that the dimensional dependence is sharp. In fact~\eqref{deffl} is obtained in~\cite{DEFFL2024} by taking the limit as $n\to+\infty$, after a rescaling by $\sqrt n$.
\smallskip
\item For $p \in (1,2^\star)$ the stability issue for the subcritical family of inequalities~\eqref{GNSm} and~\eqref{GNS-log} has been completely solved in~\cite{Frank_2022,brigati2022logarithmic}, with the caveat that the stability term degenerates on a $n$-dimensional subspace. Analogous stability estimates have been established for the subcritical family~\eqref{beckner} in~\cite{MR4703671}.
\end{itemize}

\subsubsection{The Euclidean case}

Let us briefly observe that~\eqref{LSI-G} and~\eqref{LSI-E} are equivalent, up to the issue that the two inequalities are formulated in two different spaces (and there is a cancellation of the second moments in proving the Euclidean form from the Gaussian form of the inequality, as already remarked in~\cite{MR1132315}).
However, by density, the stability result~\eqref{deffl} translates into an analogous estimate for~\eqref{LSI-E}: see for instance~\cite[Corollary~4.4]{DEFFL2025}.

\section{Examples of instability}\label{Sec:Instability}

In this last section we collect some observations on counter-examples in various norms.

\subsection{Known counter-examples}

The first observation of instability of $\delta[v]$ with respect to the Wasserstein distance~$\mathrm{W}_2$ appears in~\cite{MR3271181}. The authors note that if such a stability estimate held for all functions, it would imply an improvement of the optimal constant in the logarithmic Sobolev inequality in the form~\eqref{LSI-G}, a contradiction. The first explicit counter-example was later constructed in~\cite{kim2018} (and later in~\cite{MR4455233}): there is a sequence $(v_n)_{n\in\N}$ for which
\[
\lim_{n\to\infty}\delta[v_n]=0\,,\quad\liminf_{n\to\infty}\mathrm{W}_2^2\big(|v_n|^2\,dx, d\gamma\big)>0\quad\mbox{and}\quad \liminf_{n\to\infty}\nrm{v_n-1}2^2>0\,.
\]
The results presented in~\cite{MR4455233} and the simplified version in~\cite{Eldan2020} are primarily based on the observation that one can construct minimizing sequences for~\eqref{LSI-G}, for which the second moment becomes arbitrarily large. Crucially, the deficit $\delta[v]$ is insensitive to the second moment, an insight made precise through a computation by E.~Carlen in~\cite{MR1132315}, whereas the $\mathrm{W}_2$ distance is highly sensitive to it.

The $\mathrm H^1$ instability of~\eqref{LSI-G} was pointed out by E.~Indrei in~\cite{Indrei_2023}. The author also clarified the role of moments (see Corollary~\ref{Cor:Moments}) by constructing a sequence $(v_n)_{n\in\N}$ such that
\begin{equation*}
\lim_{n\to\infty}\delta[v_n]=0\quad\mbox{and}\quad \liminf_{n\to\infty}\nrmg{\nabla v_n}2^2=+\infty\,.
\end{equation*}

\subsection{A counter-example to \texorpdfstring{$\dot{\mathrm H}^1$}{dotH1} stability}\label{Sec:Instability2}

Here we prove that the examples constructed in~\cite{MR4455233,MR4305006} also provide an example of instability in the $\dot{\mathrm H}^1(\R^d,dx)$-distance. We recall that $\mathcal M$ denotes the set of optimisers of~\eqref{LSI-G} defined in Section~\ref{Sec:ImprovedInequalities}.
\smallskip\begin{proposition}\label{counterxmp}
Let $d\ge1$. For all $a>0$, there exists a sequence $(v_n)_{n\in\N}$ of functions in $\mathrm H^1(\R^d,d\gamma)$ such that $\|v_n\|_{\mathrm L^2(\R,d\gamma)}=1$ and
\begin{subequations}
\begin{align}
&\int_{\R^d} x\,|v_n|^2\,d\gamma=0\,,\quad\lim_{n \to \infty}\,\int_{\R^d} |x|^2\,|v_n|^2\,d\gamma=d + a\,,\label{counterexp-a} \\
&\lim_{n \to \infty}\delta[v_n]=\lim_{n \to \infty} \( \int_{\R^d} |\nabla v_n|^2\,d\gamma - \frac12 \int_{\R^d} |v_n|^2\,\log \big(|v_n|^2\big)\,d\gamma\)=0\,,\label{counterexp-b} \\
&\liminf_{n \to \infty} \quad \inf_{w \in \mathcal M}\,\|\nabla w - \nabla v_n\|^2_{\mathrm L^2(\R^d,d\gamma)} \ge\frac a4>0\,.\label{counterexp-c}
\end{align}
\end{subequations}
\end{proposition}
\begin{proof}[Proof of Proposition~\ref{counterxmp}] We notice that it is sufficient to find such a sequence in dimension $d=1$, as in higher dimensions one can consider a sum of functions depending only on one coordinate. We use a construction inspired by~\cite[Lemma~1.7]{MR4455233}. Let us consider $(g_n)_{n\in\N}$ defined for any $x\in\R$ by
\begin{equation}\label{gk}
g_n(x):=
\begin{cases}
1 &\text{ if }|x| \le \frac n2-\frac1{2\,n}\,,\\
\psi_n(|x|) &\text{ if } \frac n2-\frac1{2\,n} \le |x| \le \frac n2\,,\\
\sqrt{\varepsilon_n}\,\mathrm e^{\frac{n\,|x|}2-\frac{|n|^2}{4}} &\text{ if } |x| \ge \frac n2\,,
\end{cases}
\end{equation}
where $(\varepsilon_n)_{n\in\N}$ is a sequence such that $\lim_{n\to\infty}2\,\varepsilon_n\,n^2=a$, and $\psi_n$ is a cut-off function such that $\psi_n(\frac n2-\frac1{2\,n})=1$ and $\psi_n(\frac n2)=\sqrt{\varepsilon_n}$. We finally set $v_{n,a}=g_n/\|g_n\|_{\mathrm L^2(\R,d\gamma)}$. By construction, we have that $\irdg{|v_n|^2}=1$, and $\irdg{x\,|v_n|^2}=0$ since $v_n(x)=v_n(-x)$. By symmetry we also have that
\be{4.1}
\frac12\,\|g_n\|_{\mathrm L^2(\R,d\gamma)}^2=\int_0^{\frac n2-\frac1{2\,n}}\,\gamma(x)\,dx + \int_{\frac n2-\frac1{2\,n}}^{\frac n2} |\psi_n(x)|^2\,\gamma(x)\,dx + \varepsilon_n\,\int_\frac n2^\infty \mathrm e^{n\,x-\frac{|n|^2}2}\,\gamma(x)\,dx
\ee
and
\[
\int_0^{\frac n2-\frac1{2\,n}}|g_n|^2\,\gamma(x)\,dx=\int_{-\frac n2+\frac1{2\,n}}^0|g_n|^2\,\gamma(x)\,dx=\frac12-\Phi\(-\frac n2+\frac1{2\,n}\)
\]
where $\Phi$ is the \emph{normal cumulative function} $ \Phi(x):=\int_{-\infty}^x\,\gamma(x)\,dx $. By completing the square, we find that
\be{4.2}
\int_\frac n2^\infty \mathrm e^{n\,x-\frac{|n|^2}2}\,d\gamma=\int_\frac n2^\infty \mathrm e^{-\frac{|x-n|^2}2} \frac{\,dx}{\sqrt{2\,\pi}}=\int_{-\frac n2}^\infty \mathrm e^{-\frac{s^2}2} \frac{ds}{\sqrt{2\,\pi}}=1 -\Phi\left(-\frac{n}2\right)\,.
\ee
By combining~\eqref{4.1} and~\eqref{4.2}, we find
\[
\|g_n\|_{\mathrm L^2(\R,d\gamma)}^2=1+ 2\,\varepsilon_n + o(\varepsilon_n^2)\,.
\]
This relies on
\[
\int_{\frac n2-\frac1{2\,n}}^{\frac n2}\,|\psi_n(x)|^2\,d\gamma \le \frac1{2\,n}\,\gamma\(\frac n2-\frac1{2\,n}\)=o(\varepsilon_n^2)
\]
using $|\psi_n|\le1$, and
\[
\int_0^{\frac n2-\frac1{2\,n}} \gamma(x)\,dx=\int_{-\frac n2+\frac1{2\,n}}^0 \gamma(x)\,dx=\frac12-\Phi\(-\frac n2+\frac1{2\,n}\)=\frac12+o(\varepsilon_n^2)\,.
\]
A similar computation also shows that
\[
\int_{\mathbb{R}} |x|^2\,|v_n|^2\,d\gamma=1+ 2\,\varepsilon_n\,n^2 + 2\,\varepsilon_n+ o\left(\varepsilon_n\right)\rightarrow 1+a\quad\mbox{as}\quad n\to\infty\,,
\]
which completes the proof of~\eqref{counterexp-a}. From the definition~\eqref{gk} we find that
\[
\|v_n'\|_{\mathrm L^2(\R,d\gamma)}^2=\frac2{\|g_n\|_{\mathrm L^2(\R,d\gamma)}^2}\left(\int_{\frac n2-\frac1{2\,n}}^\frac n2 \left|\psi_n'(x)\right|^2\,d\gamma + \frac14\,{\varepsilon_n\,n^2}\left(1-\Phi\(-\frac{n}2\)\right)\right)
\]
and $\mathcal{E}[v_n]=\irdg{|v_n|^2\,\log |v_n|^2}$ is estimated by

\begin{align*}
\mathcal E[v_n]=\frac2{\|g_n\|_{\mathrm L^2(\R,d\gamma)}^2}\,\biggl(&\int_{\frac n2-\frac1{2\,n}}^\frac n2 |\psi_n(x)|^2\log|\psi_n(x)|^2\,d\gamma \\ &\qquad+\,\varepsilon_n\left( \log \varepsilon_n + \frac12\,{n^2} \right) \left(1-\Phi\(-\frac{n}2\) \right) - n\,\varepsilon_n\,\gamma\(-\frac{n}2\)\biggr)\\
&-\,2\,\varepsilon_n + o(\varepsilon_n)\,,
\end{align*}
so that
\[
\delta[v_n]=\frac1{\|g_n\|_{\mathrm L^2(\R,d\gamma)}^2}\,\big(\varepsilon_n\,|\log\varepsilon_n |+ o(\varepsilon_n)\big) + \varepsilon_n + o(\varepsilon_n)\,,
\]
which yields~\eqref{counterexp-b}. To prove~\eqref{counterexp-c}, let us establish that
\[
\inf_{w\in\mathcal M} \|v_n'-w'\|_{\mathrm L^2(\R,d\gamma)}^2\ge \frac{\varepsilon_n\,n^2}{2\,\|g_n\|_{\mathrm L^2(\R,d\gamma)}^2}\left(1-\Phi\(-\frac{n}2\)\right)\rightarrow \frac a4>0\quad\mbox{as}\quad n\to\infty\,.
\]
Let $w\in\mathcal M$: there exists $b$ and $c\in\R$ such that $w=c\,\mathrm e^{\frac{b\,x}2-\frac{b^2}{4}}$. Then $w'(x)=c\,\frac{b}2\,\mathrm e^{\frac{b\,x}2-\frac{b^2}{4}}$, and we distinguish three cases:
\\[4pt]
$\bullet$ If $b\,c=0$, then $w'=0$, so
\[
\|v_n'-w'\|_{\mathrm L^2(\R,d\gamma)}^2\ge \int_\frac n2^\infty |v_n'(x)|^2\,d\gamma=\frac{\varepsilon_n\,n^2}{2\,\|g_n\|_{\mathrm L^2(\R,d\gamma)}^2}\left(1-\Phi\(-\frac{n}2\)\right)\,.
\]
$\bullet$ Assume now that $b\,c<0$. For $x>n/2$ we have that
\[
v_n'(x)-w'(x)=\frac{n\,\varepsilon_n\,}{2\,\|g_n\|_{\mathrm L^2(\R,d\gamma)}}\,\mathrm e^{\frac{n\,x}2-\frac{n^2}{4}} - b\,c\,\mathrm e^{\frac{b\,x}2-\frac{b^2}{4}}=v_n'(x)+|b\,c|\,\mathrm e^{\frac{b\,x}2-\frac{b^2}{4}}\,,
\]
that is, for $x>n/2$, the functions $v_n'$ and $|b\,c|\,\mathrm e^{\frac{b\,x}2-\frac{b^2}{4}}$ have the same sign and are both positive. As a consequence, we have
\[
\|v_n'-w'\|_{\mathrm L^2(\R,d\gamma)}^2
\ge \int_\frac n2^\infty |v_n'(x)|^2\,d\gamma=\frac{\varepsilon_n\,n^2}{2\,\|g_n\|_{\mathrm L^2(\R,d\gamma)}^2}\left(1-\Phi\(-\frac{n}2\)\right)\,.
\]
This lower bound is uniform in $b$ and $c$. We take the infimum in $\|v_n'-w'\|_{\mathrm L^2(\R,d\gamma)}$ and obtain the sought inequality.
\\[4pt]
$\bullet$ The case $b\,c>0$ can be dealt similarly to the case $b\,c<0$ but in the interval $(-\infty, -n/2)$.
\end{proof}

\bigskip\noindent{\bf Acknowledgements.}
This work has been supported by the Project~\emph{Conviviality} (ANR-23-CE40-0003) of the French National Research Agency. G.B.~has received funding from the European Union’s Horizon 2020 research and innovation programme under the Marie Skłodowska-Curie grant agreement No 101034413. The authors thank a referee for a careful reading and suggestions which result in a significant improvement of the manuscript.\\[4pt]
{\scriptsize\copyright\,\the\year\ by the authors. This paper may be reproduced, in its entirety, for non-commercial purposes. \burlalt{https://creativecommons.org/licenses/by/4.0/legalcode}{CC-BY 4.0}}
\bigskip

\end{document}